\documentclass[11pt,a4paper]{amsart}
\usepackage{amsthm,amsfonts,amsfonts,amsmath,amssymb}
\usepackage[colorlinks=true,citecolor=black,linkcolor=black,urlcolor=blue]{hyperref}
\usepackage{iwona} 
\usepackage[numbers,sort&compress]{natbib}
\usepackage[lmargin=33mm,rmargin=33mm,tmargin=33mm,bmargin=33mm]{geometry}

\renewcommand{\baselinestretch}{1.1}
\allowdisplaybreaks

\newcommand{\doi}[1]{doi:\,\href{http://dx.doi.org/#1}{#1}}

\renewcommand{\geq}{\geqslant}
\renewcommand{\leq}{\leqslant}

\theoremstyle{plain}
\newtheorem{theorem}{Theorem}
\newtheorem{proposition}[theorem]{Proposition}
\newtheorem{conjecture}[theorem]{Conjecture}
\newtheorem{lemma}[theorem]{Lemma}

\newcommand{\floor}[1]{\ensuremath{\protect\lfloor#1\rfloor}}
\DeclareMathOperator{\polylog}{polylog}
\DeclareMathOperator{\GP}{GP}
\DeclareMathOperator{\Z}{\mathbb{Z}}

\begin{document}

\title{On the general position subset selection problem}
\date{\today}

\author[]{Michael~S.~Payne}
\author[]{David~R.~Wood}

\address{
\newline Department of Mathematics and Statistics, 
\newline The University of Melbourne
\newline Melbourne, Australia}
\email{m.payne3@pgrad.unimelb.edu.au}
\email{woodd@unimelb.edu.au}

\thanks{Michael Payne is supported by an Australian Postgraduate Award from the Australian Government. 
David Wood is supported by a QEII Research Fellowship from the Australian Research Council.}

\begin{abstract}
Let $f(n,\ell)$ be the maximum integer such that every set of $n$ points in the plane with at most $\ell$ collinear contains a subset of $f(n,\ell)$ points with no three collinear. 
First we prove that if $\ell \leq O(\sqrt{n})$ then  $f(n,\ell)\geq \Omega(\sqrt{\frac{n}{\ln \ell}})$. 
Second we prove that if $\ell \leq O(n^{(1-\epsilon)/2})$ then $f(n,\ell) \geq \Omega(\sqrt{n\log_\ell n})$, which implies all previously known lower bounds on $f(n,\ell)$ and improves them when $\ell$ is not fixed.
A more general problem is to consider subsets with at most $k$ collinear points in a point set with at most $\ell$ collinear. We also prove analogous results in this setting.
\end{abstract}

\maketitle

\section{Introduction}

A set of points in the plane is in \emph{general position} if it contains no three collinear points.
The general position subset selection problem asks, given a finite set of points in the plane with at most $\ell$ collinear, how big is the largest subset in general position? 
That is, determine the maximum integer $f(n,\ell)$ such that every set of $n$ points in the plane with at most $\ell$ collinear contains a subset of $f(n,\ell)$ points in general position. Throughout this paper we assume $\ell \geq 3$.

The problem was originally posed by Erd\H os, first for the case $\ell=3$~\cite{Erdos86}, and later in a more general form~\cite{Erdos88}.
F\"uredi~\cite{furedi} showed that the density version of the Hales--Jewett theorem~\cite{furstkatz} implies that $f(n,\ell)\leq o(n),$ and that a result of Phelps and R\"odl~\cite{phelprod} on independent sets in partial Steiner triple systems implies that $$f(n,3) \geq \Omega(\sqrt{n\ln n}).$$
Until recently, the best known lower bound for $\ell\geq 4$ was $f(n,\ell)\geq\sqrt{2n/(\ell-2)}$, proved by a greedy selection algorithm. 
Lefmann~\cite{lefmann} showed that for fixed $\ell$, $$f(n,\ell) \geq \Omega_\ell(\sqrt{n\ln n}).$$ 
(In fact, his results are more general, see Section~\ref{gens}.)

In relation to the general position subset selection problem (and its relatives), Bra\ss, Moser and Pach~\cite[page 318]{BMP} write, ``To make any further progress, one needs to explore the geometric structure of the problem.'' 
We do this by using the Szemeredi--Trotter theorem \cite{szemtrot}. 

We give improved lower bounds on $f(n,\ell)$ when $\ell$ is not fixed, with the improvement being most significant for values of $\ell$ around $\sqrt{n}$.
Our first result (Theorem~\ref{mainthm}) says that if $\ell \leq O(\sqrt{n})$ then  $f(n,\ell)\geq \Omega(\sqrt{\frac{n}{\ln \ell}})$. 
Our second result (Theorem~\ref{smallell}) says that if $\ell \leq O(n^{(1-\epsilon)/2})$ then $f(n,\ell) \geq \Omega(\sqrt{n\log_\ell n})$. Note that for fixed $\ell$, this implies the previously known lower bound on $f(n,\ell)$.

In Section~\ref{gens} we consider a natural generalisation of the general position subset selection problem. Given $k<\ell$, Erd\H os~\cite{Erdos88} asked for the maximum integer $f(n,\ell,k)$ such that every set of $n$ points in the plane with at most $\ell$ collinear contains a subset of $f(n,\ell,k)$ points with at most $k$ collinear. Thus $f(n,\ell) = f(n,\ell,2)$. We prove results similar to Theorems~\ref{mainthm} and~\ref{smallell} in this setting too.

\section{Results}

Our main tool is the following lemma.
\begin{lemma}\label{triples}
Let $P$ be a set of $n$ points in the plane with at most $\ell$ collinear.
Then the number of collinear triples in $P$ is at most $c(n^2\ln\ell +\ell^2n)$ for some constant $c$.
\end{lemma}
\begin{proof}
For $2\leq i\leq\ell$, let $s_i$ be the number of lines containing exactly $i$ points in $P$. 
A well-known corollary of the Szemeredi--Trotter theorem \cite{szemtrot} states that for some constant $c\geq 1$, for all $i\geq 2$,
  $$\sum_{j\geq i}s_j \leq c\left(\frac{n^2}{i^3}+\frac{n}{i}\right).$$
Thus the number of collinear triples is
$$\sum_{i=2}^{\ell} \binom{i}{3} s_i 
 \leq
\sum_{i=2}^{\ell} i^2\sum_{j=i}^{\ell}s_j
 \leq
\sum_{i=2}^{\ell} ci^2 \left(\frac{n^2}{i^3}+\frac{n}{i}\right)
 \leq
c\sum_{i=2}^{\ell} \left(\frac{n^2}{i}+in\right)
 \leq
c(n^2\ln\ell + \ell^2n).$$
%
\end{proof}

Note that Lefmann~\cite{Lefmann-EuJC}  proved Lemma~\ref{triples} for the case of the $\sqrt{n}\times\sqrt{n}$ grid via a direct counting argument.

To apply Lemma~\ref{triples} it is useful to consider the $3$-uniform hypergraph $H(P)$ determined by a set of points $P$, with vertex set $P$, and an edge for each collinear triple in $P$. A subset of $P$ is in general position if and only if it is an independent set in $H(P)$. The size of the largest independent set in a hypergraph $H$ is denoted $\alpha(H)$.
Spencer~\cite{Spencer} proved the following lower bound on $\alpha(H)$. 
\begin{lemma}[Spencer~\cite{Spencer}]\label{sudlem}
Let $H$ be an $r$-uniform hypergraph with $n$ vertices and $m$ edges. 
If $m<n/r$ then $\alpha(H)>n/2$.
If $m\geq n/r$ then $$\alpha(H) > \frac{r-1}{r^{r/(r-1)}\,
}\frac{n}{(m/n)^{1/(r-1)}}.$$
\end{lemma}

Lemmas~\ref{triples} and \ref{sudlem} imply our first result. 
\begin{theorem}\label{mainthm}
Let $P$ be a set of $n$ points with at most $\ell$ collinear, for some $\ell \leq O(\sqrt{n})$.
Then $P$ contains a set of $\Omega(\sqrt{\frac{n}{\ln \ell}})$ points in general position.
\end{theorem}
\begin{proof} 
Let $m$ be the number of edges in $H(P)$.
By Lemma~\ref{triples},  $m/n \leq b n \ln\ell$ for some constant $b$. 
Now apply Lemma~\ref{sudlem} with $r=3$. 
If $m<n/3$ then $\alpha(H(P))>n/2$, as required. 
Otherwise, 
$$ \alpha(H(P)) > \frac{2n}{3^{3/2}(m/n)^{1/2}}
	\geq \frac{2n}{3^{3/2}\sqrt{b n\ln\ell}}
	= \frac{2}{3\sqrt{3b}}\,\sqrt{\frac{n}{\ln\ell}}.$$
\end{proof}

Theorem~\ref{mainthm} answers, up to a logarithmic factor, a symmetric Ramsey style version of the general position subset selection problem posed by Gowers~\cite{mathoverflow}. 
He asked for the minimum integer $\GP(q)$ such that every set of at least $\GP(q)$ points in the plane contains $q$ collinear points or $q$ points in general position.
Gowers noted that $ \Omega(q^2) \leq \GP(q) \leq O(q^3)$. 
Theorem~\ref{mainthm} with $\ell=q-1$ and $n=\GP(q)$ implies that $\Omega(\sqrt{\GP(q)/\ln (q-1)}) \leq q $ and so $\GP(q) \leq O(q^2 \ln q) $.

The bound $  \GP(q) \geq \Omega(q^2)$ comes from the $q \times q$ grid, which contains no $q+1$ collinear points, and no more than $2q+1$ in general position, since each row can have at most $2$ points.
Determining the maximum number of points in general position in the $q \times q$ grid is known as the \emph{no-three-in-line problem}. 
It was first posed by Dudeney in 1917~\cite{dudeney} -- see~\cite{HJSW-JCTA75} for the best known bound and for more on its history.

As an aside, note that Pach and Sharir~\cite{pachsharir} proved a result somewhat similar to Lemma~\ref{triples} for the number of triples in $P$ determining a fixed angle $\alpha\in (0,\pi)$. Their proof is similar to that of Lemma~\ref{triples} in its use of the Szemeredi--Trotter theorem.
Also, Elekes~\cite{elekes} employed Lemma~\ref{sudlem} to prove a similar result to Theorem~\ref{mainthm} for the problem of finding large subsets with no triple determining a given angle $\alpha\in (0,\pi)$. 
Pach and Sharir and Elekes did not allow the case $\alpha = 0$, that is, collinear triples. 
This may be because their work did not consider the parameter $\ell$, without which the case $\alpha=0$ is exceptional since $P$ could be entirely collinear, and all triples could determine the same angle.

The following lemma of Sudakov~\cite[Proposition 2.3]{sudakov} is a corollary of a result by Duke, Lefmann and R\"odl~\cite{dukelefrod}.
\begin{lemma}[Sudakov~\cite{sudakov}]\label{sudprop}
Let $H$ be a $3$-uniform hypergraph on $n$ vertices with $m$ edges. Let $t \geq \sqrt{m/n}$ and suppose there exists an $\epsilon >0$ such that the number of edges containing any fixed pair of vertices of $H$ is at most $t^{1-\epsilon}$. Then $\alpha(H) \geq \Omega_\epsilon \left( \frac{n}{t}\sqrt{\ln t} \right)$. 
\end{lemma}

Lemmas~\ref{triples} and~\ref{sudprop} can be used to prove our second result. 
\begin{theorem}\label{smallell}
Fix constants $\epsilon>0$ and $d>0$. 
Let $P$ be a set of $n$ points in the plane with at most $\ell$ collinear points, where 
$3\leq \ell\leq (dn)^{(1-\epsilon)/2}$. Then $P$ contains a set of  $\Omega(\sqrt{n \log_\ell n})$ points in general position.
\end{theorem}
\begin{proof}
Let $m$ be the number of edges in $H(P)$. 
By Lemma~\ref{triples}, for some constant $c\geq 1$, 
$$m\leq 
c\ell^2n+ cn^2\ln\ell
<
cdn^2+ c n^2\ln\ell
\leq
(d+1)c n^2\ln\ell.$$ 
Define $t:= \sqrt{(d+1)cn\ln\ell}$. 
Thus $t \geq \sqrt{m/n}$. 
Each pair of vertices in $H$ is in less than $\ell$ edges of $H$, and 
$$\ell\leq (dn)^{(1-\epsilon)/2} < ((d+1)c n \ln \ell)^{(1-\epsilon)/2}=t^{1-\epsilon}.$$
Thus the assumptions in  Lemma~\ref{sudprop} are satisfied. 
So $H$ contains an independent set of size 
$\Omega(\frac{n}{t}\,\sqrt{\ln t})$. 
Moreover, 
\begin{align*}
\frac{n}{t}\,\sqrt{\ln t} 
&=
\sqrt{\frac{n}{(d+1)c\ln\ell}}\;\sqrt{\ln \sqrt{(d+1)cn\ln\ell}} \\
&\geq 
 \sqrt{\frac{n}{(d+1)c\ln\ell}}\;\sqrt{\frac{1}{2}\ln n}\\
&= 
 \sqrt{\frac{1}{2(d+1)c}}\;
  \sqrt{\frac{n\ln n}{\ln\ell}}\\
&= \Omega(\sqrt{n\log_\ell n}).
\end{align*}
Thus $P$ contains a subset of
$\Omega(\sqrt{n \log_\ell n})$ points in general position. 
\end{proof}

\section{Generalisations}\label{gens}

In this section we consider the function $f(n,\ell,k)$ defined to be the maximum integer such that every set of $n$ points in the plane with at most $\ell$ collinear contains a subset of $f(n,\ell,k)$ points with at most $k$ collinear, where $k < \ell$.

Bra\ss~\cite{brass} considered this question for fixed $\ell=k+1$, and showed that $$ o(n) \geq f(n,k+1, k) \geq \Omega_k(n^{(k-1)/k}(\ln n)^{1/k}).$$
This can be seen as a generalisation of the results of F\"uredi~\cite{furedi} for $f(n,3,2)$. 
As in F\"uredi's work, the lower bound comes from a result on partial Steiner systems~\cite{rodsin}, and the upper bound comes from the density Hales--Jewett theorem~\cite{furstkatz2}.
Lefmann~\cite{lefmann} further generalised these results for fixed $\ell$ and $k$ by showing that $$f(n,\ell,k) \geq \Omega_{\ell,k}(n^{(k-1)/k}(\ln n)^{1/k}).$$
The density Hales--Jewett theorem also implies the general bound $f(n,\ell,k) \leq o(n)$.


The result of Lefmann may be generalised to include the dependence of $f(n,\ell,k)$ on $\ell$ for fixed $k\geq3$, analogously to Theorems~\ref{mainthm} and \ref{smallell} for $k=2$.
The first result we need is a generalisation of Lemma~\ref{triples}. It is proved in the same way.
\begin{lemma}\label{ktuples}
Let $P$ be a set of $n$ points in the plane with at most $\ell$ collinear.
Then, for $k\geq 4$, the number of collinear $k$-tuples in $P$ is at most $c(\ell^{k-3}n^2 + l^{k-1}n)$ for some absolute constant $c$. \qed
\end{lemma}



Lemmas~\ref{sudlem} and \ref{ktuples} imply the following theorem which is proved in the same way as Theorem~\ref{mainthm}.
\begin{theorem}\label{maink}
If $k\geq 3$ is fixed and $\ell \leq O(\sqrt{n})$, then $f(n,\ell,k) \geq \Omega_k\left(\frac{n^{(k-1)/k}}{\ell^{(k-2)/k}}\right)$.  \qed
\end{theorem}
For $\ell = \sqrt{n}$, Theorem~\ref{maink} implies $f(n,\sqrt{n},k) \geq \Omega_k\left(\frac{n^{(k-1)/k}}{n^{(k-2)/2k}} \right) = \Omega_k\left(n^{(2k-2 - k+2)/2k} \right) = \Omega_k(\sqrt{n})$. For $k\geq 3$ this answers completely a generalised version of Gowers' question~\cite{mathoverflow}, namely, to determine the minimum integer $\GP_k(q)$ such that every set of at least $\GP_k(q)$ points in the plane contains $q$ collinear points or $q$ points with at most $k$ collinear.
Thus $\GP_k(q) \leq O(q^2)$. 
The bound $\GP_k(q) \geq \Omega(q^2)$ comes from the following construction.
Let $m := \floor{ (q-1) / k }$
and let $P$ be the $m \times m$ grid.
Then $P$ has at most $m$ points collinear, and $m<q$. 
If $S$ is a subset of $P$ with at most $k$ collinear, then $S$ has at
most $k$ points in each row. 
So $|S|  \leq km \leq q-1$.

Theorem~\ref{smallell} can be generalised using Lemma~\ref{ktuples} and a theorem of Duke, Lefmann and R\"odl~\cite{dukelefrod} (the one that implies Lemma~\ref{sudprop}).
\begin{theorem}[Duke, Lefmann and R\"odl~\cite{dukelefrod}]\label{DLR}
Let $H$ be a $k$-uniform hypergraph with maximum degree $\Delta(H) \leq t^{k-1}$ where $t\gg k$.
Let $p_j(H)$ be the number of pairs of edges of $H$ sharing exactly $j$ vertices.
If $ p_j(H) \leq n t^{2k-j-1-\gamma}$ for $j = 2, \dots , k-1$ and some $\gamma >0$, then $\alpha(H) \geq \Omega_{k,\gamma}\left( \frac{n}{t}(\ln n)^{1/(k-1)}\right)$.
\end{theorem}


\begin{theorem}
Fix constants $d>0$ and $\epsilon \in (0,1)$. If $k\geq 3$ is fixed and $4 \leq \ell \leq dn^{(1-\epsilon)/2}$ then $$f(n,\ell,k) \geq \Omega_k\left(\frac{n^{(k-1)/k}}{\ell^{(k-2)/k}}(\ln n)^{1/k}\right).$$
\end{theorem}


\begin{proof}
Given a set $P$ of $n$ points with at most $\ell$ collinear, a subset with at most $k$ collinear points corresponds to an independent set in the $(k+1)$-uniform hypergraph $H_{k+1}(P)$ of collinear $(k+1)$-tuples in $P$. 
By Lemma~\ref{ktuples}, the number of edges in  $H_{k+1}(P)$ is $ m \leq c(n^2\ell^{k-2} + nl^{k})$ for some constant $c$. 
The first term dominates since $\ell \leq o(\sqrt{n})$.
For $n$ large enough, $ m/n \leq 2c n\ell^{k-2}$.

To limit the maximum degree of $H_{k+1}(P)$, discard vertices of degree greater than $2(k+1)m/n$.
Let $\tilde{n}$ be the number of such vertices.
Considering the sum of degrees, $(k+1)m \geq \tilde{n}2(k+1)m/n$, and so $\tilde{n} \leq n/2$.
Thus discarding these vertices yields a new point set $P'$ such that $|P'|\geq n/2$ and $\Delta(H_{k+1}(P')) \leq 4(k+1)c n\ell^{k-2}$.
Note that an independent set in $H_{k+1}(P')$ is also independent in $H_{k+1}(P)$.

Set $t:= (4(k+1)cn\ell^{k-2})^{1/k}$, so $m\leq \frac{1}{2(k+1)}nt^k$ and $\Delta(H_{k+1}(P')) \leq t^k$, as required for Theorem~\ref{DLR}.
By assumption, $\ell \leq dn^{(1-\epsilon)/2}$.
Thus 
$$\ell \leq d \left( \frac{t^k\ell^{2-k}}{4(k+1)c} \right)^{\frac{1-\epsilon}{2}}.$$ 
Hence $\ell^{\frac{2}{1-\epsilon} +k-2}  \leq \frac{d^{2/(1-\epsilon)}t^k}{4(k+1)c}$, 
implying $\ell  \leq C_1(k)t^{\frac{k}{\frac{2}{1-\epsilon} +k-2}} 
 = C_1(k) t^{\frac{1-\epsilon}{1-\epsilon +\frac{2k}{\epsilon} }}$
for some constant $C_1(k)$.
Define $\epsilon' := 1- \frac{1-\epsilon}{1-\epsilon +\frac{2k}{\epsilon} }$, so $\epsilon' >0 $ (since $\epsilon <1$) and $\ell \leq C_1(k)t^{1-\epsilon'}$. 
To bound $p_j(H_{k+1}(P'))$ for $j=2, \dots, k$, first choose one edge (which determines a line), then choose the subset to be shared, then choose points from the line to complete the second edge of the pair. Thus for $\gamma := \epsilon'/2$ and sufficiently large $n$,
\begin{align*}
p_j(H_{k+1}(P')) & \leq m \binom{k+1}{j} \binom{\ell -k-1}{k+1-j} \\
& \leq C_2(k)nt^k \ell^{k+1-j} \\
& \leq C_2(k)(C_1(k))^{k+1-j}nt^kt^{(1-\epsilon')(k+1-j)} \\
& \leq nt^{2(k+1)-j-1-\gamma}.
\end{align*}
Hence the second requirement of Theorem~\ref{DLR} is satisfied.
%
Thus
\begin{align*}
\alpha(H_{k+1}(P')) & \geq \Omega_{k,\epsilon}\left( \frac{n}{t}(\ln t)^{1/k}\right) \\
& \geq \Omega_{k,\epsilon}\left(\frac{n^{(k-1)/k}}{\ell^{(k-2)/k}}\left(\ln( (n\ell^{k-2})^{1/k} )\right)^{1/k}\right) \\
& \geq \Omega_{k,\epsilon}\left(\frac{n^{(k-1)/k}}{\ell^{(k-2)/k}}(\ln n)^{1/k}\right).
\end{align*}\end{proof}


\section{Conjectures}

Theorem~\ref{maink} suggests the following conjecture, which would completely answer Gowers' question~\cite{mathoverflow}, showing that $\GP(q)=\Theta(q^2)$.
It is true for the $\sqrt{n}\times\sqrt{n}$ grid \cite{HJSW-JCTA75,Erdos51}. 
\begin{conjecture}\label{nolog}
$f(n, \sqrt{n}) \geq \Omega(\sqrt{n})$.
\end{conjecture}

A natural variation of the general position subset selection problem is to colour the points of $P$ with as few colours as possible, such that each colour class is in general position. 
An easy application of the Lov\'asz Local Lemma shows that under this requirement, $n$ points with at most $\ell$ collinear are colourable with $O(\sqrt{\ell n})$ colours. 
The following conjecture would imply Conjecture~\ref{nolog}.
It is also true for the $\sqrt{n}\times\sqrt{n}$ grid \cite{Wood04}.

\begin{conjecture}\label{colours}
Every set $P$ of $n$ points in the plane with at most $\sqrt{n}$ collinear  
can be coloured with $O(\sqrt{n})$ colours such that each colour class is in general position.
\end{conjecture}

The following proposition is somewhat weaker than Conjecture~\ref{colours}.

\begin{proposition}\label{colprop}
Every set $P$ of $n$ points in the plane with at most $\sqrt{n}$ collinear  
can be coloured with $O(\sqrt{n} \ln^{3/2} n)$ colours such that each colour class is in general position.
\end{proposition}

\begin{proof}
Colour $P$ by iteratively selecting a largest subset in general position and giving it a new colour.
Let $P_0 := P$.
Let $C_i$ be a largest subset of $P_i$ in general position and let $P_{i+1} := P_i \setminus C_i$.
Define $n_i := |P_i|$.
Applying Lemma~\ref{triples} to $P_i$ shows that $H(P_i)$ has $O(n_i^2\ln\ell + \ell^2n_i)$ edges.
Thus the average degree of $H(P_i)$ is at most $O(n_i\ln\ell + \ell^2)$ which is $O(n\ln n)$ since $n_i\leq n$ and $\ell \leq \sqrt{n}$.

Applying Lemma~\ref{sudlem} gives $|C_i| = \alpha(H(P_i)) > cn_i/ \sqrt{n \ln n}$ for some constant $c>0$. 
Thus $n_i \leq n(1-c/\sqrt{n \ln n})^i$.
It is well known (and not difficult to show) that if a sequence of numbers $m_i$ satisfies $m_i \leq m(1-1/x)^i$ for some $x>1$ 
and if $j > x\ln m$,  then $m_j \leq 1$.
Hence if $k \geq \sqrt{n \ln n} \ln n /c$ then $n_k \leq 1$, so the number of colours used is $O(\sqrt{n} \ln^{3/2} n)$.
\end{proof}

The problem of determining the correct asymptotics of $f(n,\ell)$ (and $f(n,\ell,k)$) for fixed $\ell$ remains wide open. 
The Szemeredi--Trotter theorem 
is essentially tight for the $\sqrt{n}\times\sqrt{n}$ grid~\cite{pachtoth}, 
but says nothing for point sets with bounded collinearities.
For this reason, the lower bounds on $f(n,\ell)$ for fixed $\ell$ remain essentially combinatorial. 
Finding a way to bring geometric information to bear in this situation is an interesting challenge.

\begin{conjecture}\label{hard}
If $\ell$ is fixed, then 
$f(n,\ell) \geq \Omega_{\ell}(n/\polylog(n))$.
\end{conjecture}

The point set that gives the upper bound $f(n,\ell) \leq o(n)$ (from the density Hales--Jewett theorem) is the generic projection to the plane of 
the ${\floor{\log_\ell n}}$-dimensional $\ell \times \ell \times \dots \times \ell$ integer lattice (henceforth $[\ell]^d$ where $d := \floor{\log_\ell(n)}$).
The problem of finding large general position subsets in this point set for $\ell=3$ 
is known as Moser's cube problem~\cite{moser, polymath}, and the best known asymptotic lower bound is $\Omega(n / \sqrt{\ln n})$~\cite{chvatal, polymath}.

In the colouring setting, the following conjecture is equivalent to 
Conjecture~\ref{hard} by an argument similar to that of Proposition~\ref{colprop}.

\begin{conjecture}\label{hardcol}
For all fixed $\ell\geq3$, every  set of $n$ points in the plane with at most $\ell$ collinear  
can be coloured with $O_\ell(\polylog(n))$ colours such that each colour class is in general position.
\end{conjecture}

Conjecture~\ref{hardcol} is true for $[\ell]^d$,
 which can be coloured with $O(d^{\ell-1})$ 
colours as follows. 
For each $x \in [\ell]^d$, define a signature vector in $\Z^\ell$ whose entries are the number of entries in $x$ equal to $1, 2, \dots \ell$.
The number of such signatures is the number of partitions of $d$ into at most $\ell$ parts, which is $O(d^{\ell-1})$.
Give each set of points with the same signature its own colour.
To see that this is a proper colouring, suppose that $\{a,b,c\}\subset [\ell]^d$ is a monochromatic collinear triple, with $b$ between $a$ and $c$.
Permute the coordinates so that the entries of $b$ are non-decreasing.
Consider the first coordinate $i$ in which $a_i$, $b_i$ and $c_i$ are not all equal.
Then without loss of generality, $a_i<b_i$.
But this implies that $a$ has more entries equal to $a_i$ than $b$ does, contradicting the assumption that the signatures are equal.


\subsection*{Acknowledgements}
Thanks to Timothy Gowers for posting his motivating problem on MathOverflow~\cite{mathoverflow}.
Thanks to Moritz Schmitt and Louis Theran for interesting discussions, and for making us aware of reference~\cite{mathoverflow}.
Thanks to J\'{a}nos Pach for pointing out references \cite{lefmann} and \cite{pachsharir}.

\bibliographystyle{myNatbibStyle}
\bibliography{references}

\end{document}